\documentclass[12pt]{amsart}
\usepackage{fullpage,amssymb,amsthm,amsmath,url,stmaryrd}
\usepackage[all]{xy}

\newtheorem{theorem}{Theorem}
\newtheorem{lemma}[theorem]{Lemma}
\newtheorem{prop}[theorem]{Proposition}
\theoremstyle{definition}
\newtheorem{hypothesis}[theorem]{Hypothesis}
\newtheorem{defn}[theorem]{Definition}
\newtheorem{example}[theorem]{Example}
\newtheorem{remark}[theorem]{Remark}
\numberwithin{theorem}{section}
\numberwithin{equation}{theorem}

\newcommand{\bA}{\mathbf{A}}
\newcommand{\be}{\mathbf{e}}
\newcommand{\bv}{\mathbf{v}}

\newcommand{\QQ}{\mathbb{Q}}
\newcommand{\RR}{\mathbb{R}}
\newcommand{\ZZ}{\mathbb{Z}}

\newcommand{\calF}{\mathcal{F}}

\newcommand{\calO}{\mathcal{O}}

\newcommand{\frakm}{\mathfrak{m}}
\newcommand{\frako}{\mathfrak{o}}
\newcommand{\frakp}{\mathfrak{p}}

\DeclareMathOperator{\Frac}{Frac}
\DeclareMathOperator{\Gr}{Gr}
\DeclareMathOperator{\rinf}{inf}
\DeclareMathOperator{\Mod}{\mathbf{Mod}}
\DeclareMathOperator{\ff}{ff}
\DeclareMathOperator{\Spa}{Spa}
\DeclareMathOperator{\Spec}{Spec}
\DeclareMathOperator{\Tor}{Tor}
\DeclareMathOperator{\Vect}{\mathbf{Vec}}

\begin{document}

\title{Some ring-theoretic properties of $\bA_{\rinf}$}
\author{Kiran S. Kedlaya}
\date{June 5, 2019}
\thanks{Supported by NSF (grant DMS-1501214), UC San Diego (Warschawski Professorship), Guggenheim Fellowship (fall 2015). Some of this work was carried out at MSRI during the fall 2014 research program ``New geometric methods in number theory and automorphic forms'' supported by NSF grant DMS-0932078. Thanks to Jaclyn Lang, Judith Ludwig, and Peter Scholze for additional feedback.}

\begin{abstract}
The ring of Witt vectors over a perfect valuation ring of characteristic $p$, often denoted $\bA_{\rinf}$, plays a pivotal role in $p$-adic Hodge theory; for instance, Bhatt--Morrow--Scholze have recently reinterpreted and refined the crystalline comparison isomorphism by relating it to a certain $\bA_{\rinf}$-valued cohomology theory. We address some basic ring-theoretic questions about $\bA_{\rinf}$, motivated by analogies with two-dimensional regular local rings. For example, we show that in most cases $\bA_{\rinf}$, which is manifestly not noetherian, is also not coherent. On the other hand, it does have the property that vector bundles over the complement of the closed point in $\Spec \bA_{\rinf}$ do extend uniquely over the puncture; moreover, a similar statement holds in Huber's category of adic spaces.
\end{abstract}

\maketitle

Throughout this paper, let $K$ be a perfect field of characteristic $p$ equipped with a nontrivial valuation $v$ (written additively), e.g., the perfect closure of $\mathbb{F}_p((t))$ with the $t$-adic valuation. 
(Note that $K = \mathbb{F}_p$ is excluded by the nontriviality condition.)
Unless otherwise specified, we do not assume that $K$ is complete.

A fundamental role is played in $p$-adic Hodge theory by the ring
$\bA_{\rinf} := W(\frako_K)$, where $\frako_K$ denotes the valuation ring of $K$ and $W$ denotes the functor of $p$-typical Witt vectors. The ring $\bA_{\rinf}$ serves as the basis for Fontaine's construction of $p$-adic period rings and the ensuing analysis of comparison isomorphisms. Recently,
Fargues has used $\bA_{\rinf}$ to give a new description of crystalline representations
via a variant of Breuil--Kisin modules \cite{fargues},
while Bhatt--Morrow--Scholze have described the crystalline comparison isomorphism via a direct construction of these modules \cite{bhatt-morrow-scholze}.

We discuss several issues germane to \cite{bhatt-morrow-scholze} regarding ring-theoretic properties of $\bA_{\rinf}$, particularly those related to the analogy between $\bA_{\rinf}$ and two-dimensional regular local rings. 
In the negative direction, the ring $\bA_{\rinf}$ is typically not coherent (Theorem~\ref{T:not coherent}); in the positive direction, vector bundles over the complement of the closed point in $\Spec(\bA_{\rinf})$ extend over the puncture
(Theorem~\ref{T:algebraic glueing local}), and similarly if the Zariski spectrum is replaced by the Huber adic spectrum (Theorem~\ref{T:adic glueing local}). 

We also discuss briefly some related questions in the case where $K$ is replaced by a more general nonarchimedean Banach ring. These are expected to pertain to a hypothetical relative version of the results of \cite{bhatt-morrow-scholze}. 

\section{Finite generation properties}
\label{sec:modules Witt rings}

\begin{defn}
A ring is \emph{coherent} if every finitely generated ideal is finitely presented. Note that an integral domain is coherent if and only if 
the intersection of any two finitely generated ideals is again finitely generated
\cite{chase}. 
\end{defn}

A result of Anderson--Watkins \cite{anderson-watkins}, building on work of J\o ndrup--Small \cite{jondrup-small} and Vasconcelos \cite{vasconcelos} (see also \cite[Theorem~8.1.9]{graz}), asserts that a power series ring over a nondiscrete valuation ring can never be coherent except possibly if the value group is isomorphic to $\RR$. Using a similar technique, we have the following. 

\begin{theorem} \label{T:not coherent}
Suppose that the value group of $K$ is not isomorphic to $\RR$. Then $\bA_{\rinf}$ is not coherent.
\end{theorem}
\begin{proof}
It suffices to exhibit elements $f,g \in \bA_{\rinf}$ such that $(f) \cap (g)$ is not finitely generated. 
Suppose first that the value group of $K$ is archimedean, i.e., the valuation $v$ can be taken to have values in $\RR$.
Since $K$ is perfect, its value group cannot be discrete, and hence must be dense in $\RR$. We can thus choose elements
$\overline{x}_0, \overline{x}_1, \ldots \in \frako_K$ such that
$v(\overline{x}_0), v(\overline{x}_1), \dots$ is a decreasing sequence with positive limit $r \notin v(\frako_K)$ and
$v(\overline{x}_0/\overline{x}_1) > v(\overline{x}_1/\overline{x}_2) > \cdots$.
Put $f := [\overline{x}_0]$ and $g := \sum_{n=0}^\infty p^n [\overline{x}_n]$.

Recall that the ring $\bA_{\rinf}$ admits a theory of \emph{Newton polygons} analogous to the corresponding theory for polynomials or power series over a valuation ring; see \cite[Definition~4.2.8]{kl1} for details. To form the Newton polygon of $g$, we take the lower convex hull of the set $\{(n, v(\overline{x}_n)): n=0,1,\dots\}$ in $\RR^2$; the slopes of this polygon are equal to
$-v(\overline{x}_n/\overline{x}_{n+1})$ for $n=0,1,\dots$.
If $h = \sum_{n=0}^\infty p^n [\overline{h}_n] \in \bA_{\rinf}$
is divisible by both $f$ and $g$, then 
on one hand, we have $h/f= \sum_{n=0}^\infty p^n [\overline{h}_n/\overline{x}_0]$,
so $v(\overline{h}_n) \geq v(\overline{x}_0)$ for all $n$;
on the other hand, the Newton polygon of $h$ must include all of the slopes of the Newton polygon of $g$, so its total width must be at least $r$. It follows that 
$v(\overline{h}_0) \geq 2v(\overline{x}_0) - r$.

Conversely, any $\overline{h}_0 \in \frako_K$ with $v(\overline{h}_0) \geq 2v(\overline{x}_0) - r$ extends to some $h \in \bA_{\rinf}$ divisible by both $f$ and $g$, e.g., by taking $h = g [\overline{h}_0]/[\overline{x}_0]$. Since $2v(\overline{x}_0) - r \notin v(\frako_K)$, it follows that the image of $(f) \cap (g)$ in $\frako_K$ is an ideal which is not finitely generated; consequently, $(f) \cap (g)$ itself cannot be finitely generated.

Suppose next that the value group of $K$ is not archimedean. We can then choose
some nonzero $\overline{x}, \overline{y} \in \frako_K$ such that for every positive integer  $n$, $\overline{x}$ is divisible by $\overline{y}^n$ in $\frako_K$.
Let $r_1, r_2, \dots$ be a decreasing sequence of elements of $\ZZ[p^{-1}]_{>0}$
whose sum diverges.
Put $f := [\overline{x}]$ and
$g := \sum_{n=0}^\infty p^n [\overline{x}/\overline{y}^{r_1 + \cdots + r_n}]$.
As above, we see that if $h = \sum_{n=0}^\infty p^n [\overline{h}_n] \in \bA_{\rinf}$ is divisible by both $f$ and $g$, then on one hand, we have $v(\overline{h}_n) \geq v(\overline{x})$ for each $n$; on the other hand, the Newton polygon of $h$
includes all of the slopes of the Newton polygon of $g$, so
its total width must exceed $r_1 + \dots + r_n$ for each $n$.
It follows that $v(\overline{h}_0) \geq v(\overline{x}) + nv(\overline{y})$ for every positive integer $n$; conversely, any $\overline{h}_0$ with this property occurs this way for $h = g[\overline{h}_0]/[\overline{x}]$. Again, this means that $(f) \cap (g)$ maps to an ideal of $\frako_K$ which is not finitely generated, so $(f) \cap (g)$ cannot itself be finitely generated.
\end{proof}

\begin{remark}
It is unclear whether the ring $\bA_{\rinf}$ fails to be coherent even if the value group of $K$ equals $\RR$, especially if we also assume that $K$ is spherically complete.
It is also unclear whether the ring $\bA_{\rinf}[p^{-1}]$ is coherent.
By contrast, with no restrictions on $K$, for every positive integer $n$ the quotient $\bA_{\rinf}/(p^n)$ is coherent \cite[Proposition~3.24]{bhatt-morrow-scholze}.
\end{remark}

\begin{remark}
Let $\frakm_K$ be the maximal ideal of $K$. In order to apply the formalism of almost
ring theory (e.g., as developed in \cite{gabber-ramero}) to the ring $\bA_{\rinf}$,
it would be useful to know that the ideal $W(\frakm_K)$ of $\bA_{\rinf}$ has the property that $W(\frakm_K) \otimes_{\bA_{\rinf}} W(\frakm_K) \to W(\frakm_K)$ is an isomorphism. We do not know whether this holds in general; 
for example, to prove that this map fails to be surjective, one would have to produce an element of $W(\frakm_K)$ which cannot be written as a finite sum of pairwise products, and we do not have a mechanism in mind for precluding the existence of such a presentation. An easier task is to produce elements of $W(\frakm_K)$ not lying in the image of the multiplication map $W(\frakm_K) \times W(\frakm_K) \to W(\frakm_K)$, as in the following example communicated to us by Peter Scholze.
\end{remark}

\begin{example}
Suppose that $v(K^\times) = \QQ$. 
We first construct a sequence $r_1, r_2, \dots$ of positive elements of $\QQ$ with sum 1 such that every infinite subsequence with infinite complement has irrational sum. To this end, take a sequence $1 = s_0, s_1, s_2, \dots$
converging to 0 sufficiently rapidly (e.g., doubly exponentially) and put $r_1 = s_0 - s_1, r_2 = s_1 - s_2, \dots$; any infinite subsequence with infinite complement can be regrouped into sums of consecutive terms, yielding another infinite sequence with rapid decay, and Liouville's criterion implies that the sum of the subsequence is irrational (and even transcendental).

Now choose $x = \sum_{n=0}^\infty p^n [\overline{x}_n] \in W(\frakm_K)$ 
with $v(\overline{x}_n) = s_n$; we check that $x \neq yz$ for all $y,z \in W(\frakm_K)$.
If the equality $x=yz$ were to hold, the Newton polygons of $y$ and $z$ together would comprise the Newton polygon of $x$; that is, each slope occurs in $xy$ with multiplicity equal to the sum of its multiplicities in the Newton polygons of $x$ and $y$. Due to the irrationality statement of the previous paragraph, this is impossible if both $y$ and $z$ have infinitely many slopes; consequently, one of the factors, say $y$, has only finitely many slopes in its Newton polygon. On the other hand, if $y = \sum_{n=0}^\infty p^n [\overline{y}_n]$, there cannot exist $c>0$ such that $v(\overline{y}_n) \geq c$ for all $n$, as otherwise we would also have $v(\overline{x}_n) \geq c$ for all $n$.
Putting these two facts together, we deduce that $v(\overline{y}_n) = 0$ for some $n$, a contradiction.
\end{example}

The following related remark was suggested by Bhargav Bhatt.
\begin{remark}
Suppose that the value group of $K$ is archimedean.
Consider the following chain of strict inclusions of ideals:
\[
0 \subset \bigcup_{\varpi \in \frakm_K} [\varpi] \bA_{\rinf} \subset W(\frakm_K) \subset (p) + W(\frakm_K)
\]
The quotients by the ideals $W(\frakm_K)$ and $(p) + W(\frakm_K)$ are the integral domains
$W(\kappa)$ and $\kappa$, where $\kappa := \frako_K/\frakm_K$ is the residue field of $K$; hence these two ideals
are prime. The ideal $\bigcup_{\varpi \in \frakm_K} [\varpi] \bA_{\rinf}$ is also prime:
it contains $x = \sum_{n=0}^\infty p^n [\overline{x}_n]$ if and only if the total multiplicity of all slopes in the Newton polygon of $x$ is strictly less than $v(\overline{x}_0)$.

The previous argument shows that the global (Krull) dimension of $\bA_{\rinf}$ is at least 3. 
In fact, one can push this further: by adapting a construction of Arnold
\cite{arnold} that produces arbitrary long chains of prime ideals within
the ring of formal power series over a nondiscrete valuation ring,
Lang--Ludwig \cite{lang-ludwig} have shown that $\bA_{\rinf}$ has infinite Krull dimension.
\end{remark}

\section{Vector bundles}
\label{sec:algebraic glueing local}

Recall that for $A$ a two-dimensional regular local ring, the restriction functor from vector bundles on $\Spec A$ (i.e., finite free $A$-modules) to vector bundles on the complement of the closed point is an equivalence of categories. This is usually shown by using the fact that a reflexive module has depth at least 2 \cite[Tag~0AVA]{stacks-project} in conjunction with the Auslander--Buchsbaum formula \cite[Tag~090U]{stacks-project} to see that every reflexive $A$-module is projective.

During the course of Scholze's 2014 Berkeley lectures documented in \cite{s-berkeley}, we explained to him an alternate proof applicable to the case of $\bA_{\rinf}$; this argument appears as \cite[Theorem~14.2.1]{s-berkeley},
and a similar argument is given in \cite[Lemma~4.6]{bhatt-morrow-scholze}. Here, we give a general version of this proof applicable in a variety of cases, which identifies the most essential hypotheses on the ring $A$.

\begin{hypothesis}
Throughout \S\ref{sec:algebraic glueing local},
let $A$ be a local ring whose maximal ideal $\frakp$ contains a non-zero-divisor $\pi$ such that $\frako := A/(\pi)$ is (reduced and) a valuation ring with maximal ideal $\frakm$. Put $L := \Frac \frako$; in the case $A = \bA_{\rinf}$ we have $L = K$.
\end{hypothesis}

\begin{defn} \label{D:valuation field}
Put $X := \Spec(A)$, $Y := X \setminus \{\frakp\}$, and $U := \Spec(A[\pi^{-1}]) \subset X$.
Let $B$ be the $\pi$-adic completion of $A_{(\pi)}$; note that within $B[\pi^{-1}]$ we have
\begin{equation} \label{eq:ring intersect}
A[\pi^{-1}] \cap B = A.
\end{equation}
Let $Z$ be the algebraic stack which is the colimit of the diagram
\[
\Spec(A[\pi^{-1}]) \leftarrow  \Spec(B[\pi^{-1}]) \rightarrow
\Spec(B).
\]
\end{defn}

\begin{lemma} \label{L:algebraic glueing1 local}
For $* \in \{X,Y,Z\}$, let $\Vect_*$ denote the category of vector bundles on $*$.
\begin{enumerate}
\item[(a)]
The pullback functor 
$\Vect_{X} \to \Vect_{Y}$ is 
fully faithful.
\item[(b)]
The pullback functor 
$\Vect_{Y} \to \Vect_{Z}$ is 
fully faithful.
\item[(c)]
For $\calF \in \Vect_*$ and $M := H^0(*, \calF)$, the adjunction morphism
$\tilde{M}|_* \to \calF$ is an isomorphism.
\end{enumerate}
\end{lemma}
\begin{proof}
For convenience, we write $\calO$ instead of $\calO_*$ hereafter.
To deduce (a), note that by \eqref{eq:ring intersect},
\[
H^0(X, \calO) = H^0(Y, \calO) = H^0(Z, \calO) = A.
\]
To deduce (b), choose $z \in A$ whose image in $A/(\pi)$ is a nonzero element of $\frakm$, so that 
\[
\Spec(A) = U \cup V, \qquad  V := \Spec(A[z^{-1})];
\]
then note that $z$ is invertible in $B$, and within $B[\pi^{-1}]$ we have
\[
A[z^{-1}, \pi^{-1}] \cap B = A[z^{-1}].
\]
To deduce (c), note that in case $* = Y$, the injectivity of the maps
\[
H^0(U, \calO) \to H^0(U \cap V, \calO), \qquad H^0(V, \calO) \to H^0(U \cap V, \calO)
\]
implies the injectivity of the maps
\[
H^0(U, \calF) \to H^0(U \cap V, \calF), \qquad H^0(V, \calF) \to H^0(U \cap V, \calF)
\]
and hence the injectivity of the maps
\[
M \to H^0(U, \calF), \qquad M \to H^0(V, \calF).
\]
It follows easily that the maps 
\[
M \otimes_{R} H^0(U, \calO) \to H^0(U, \calF),
\qquad
M \otimes_{R} H^0(V, \calO) \to H^0(V, \calF)
\]
are isomorphisms. The case $*=Z$ is similar.
\end{proof}

The following lemma is taken from \cite[Lemma~14.2.3]{s-berkeley}.
\begin{lemma} \label{L:valuation lattice}
Let $\kappa$ be the residue field of $A$, which is also the residue field of $\frako$.
Let $d$ be a nonnegative integer.
Let $N$ be an $\frako$-submodule of $L^d$. Then $\dim_{\kappa}(N \otimes_{\frako} \kappa) \leq d$, with equality if and only if $N$ is a free module of rank $d$.
\end{lemma}
\begin{proof}
By induction on $d$, we reduce to the case $d=1$. We then see that $\dim_{\kappa}(N \otimes_{\frako} \kappa)$ equals 1 if the set of valuations of elements of $N$ has a least element, in which case $N$ is free of rank 1, and 0 otherwise.
\end{proof}

\begin{lemma} \label{L:transfer generators local}
For $\calF \in \Vect_Z$ of rank $d$, if the elements
$\bv_1,\dots,\bv_d \in H^0(Z, \calF)$ generate both $H^0(U, \calF)$
and $H^0(\Spec(L), \calF)$, then they also generate $M := H^0(Z, \calF)$. 
\end{lemma}
\begin{proof}
Choose any $\bv \in M$. Since $\bv_1,\dots,\bv_d$ generate $H^0(U, \calF)$, there exists a unique tuple $(r_1,\dots,r_d)$ over $A[\pi^{-1}]$ such that $\bv = \sum_{i=1}^d r_i \bv_i$.
In particular, there exists a nonnegative integer $m$ such that $\pi^m r_1,\dots, \pi^m r_d \in A$. If $m>0$, then $\pi^m \bv$ is divisible by $\pi$ in $M$, so it maps to zero in $H^0(\Spec(L), \calF)$. Since $\bv_1,\dots,\bv_d$ form a basis of this module,
$\pi^m r_1,\dots,\pi^m r_d$ must be divisible by $\pi$ in $A$ and so $\pi^{m-1}r_1,\dots,\pi^{m-1}r_d \in A$. By induction, we deduce that $r_1,\dots,r_d \in A$. This proves the claim.
\end{proof}

\begin{lemma} \label{L:algebraic glueing3 local}
For $\calF \in \Vect_Z$ of rank $d$, the module $M := H^0(Z, \calF)$
is free of rank $d$ over $A$.
\end{lemma}
\begin{proof}
By Lemma~\ref{L:algebraic glueing1 local}(c), $M[\pi^{-1}] = H^0(U, \calF)$ is a projective $A[\pi^{-1}]$-module of rank $d$, so we can find a finite free $A[\pi^{-1}]$-module $F$ and an isomorphism $F \cong M[\pi^{-1}] \oplus P$ for some finite projective $A[\pi^{-1}]$-module $P$. 
By rescaling by a suitably large power of $\pi$, we may exhibit a basis of $F$ consisting of elements whose projections to $M[\pi^{-1}]$ all belong to $M$. This basis then gives rise to an isomorphism $F \cong F_0[\pi^{-1}]$ for $F_0$ the finite free $A$-module on the same basis. View
\[
\Gr M[\pi^{-1}] := \bigoplus_{n \in \ZZ} (M[\pi^{-1}] \cap \pi^n F_0)/(M[\pi^{-1}] \cap \pi^{n+1} F_0)
\]
as a finite projective graded module of rank $d$ over the graded ring
\[
\Gr A[\pi^{-1}] := \bigoplus_{n \in \ZZ} \pi^n A/\pi^{n+1} A \cong \frako (\!( \overline{\pi})\!),
\]
then put
\[
V := (\Gr M[\pi^{-1}]) \otimes_{\frako(\!(\overline{\pi})\!)} \kappa(\!(\overline{\pi})\!).
\]
Note that for the $\pi$-adic topology, the image of $M$ in $\Gr M[\pi^{-1}]$ is both open (because $M$ contains a set of module generators of $M[\pi^{-1}]$) and bounded (because the same holds for the dual bundle).
Consequently, the image $T$ of $M$ in $V$ is a $\kappa \llbracket \overline{\pi} \rrbracket$-sublattice of $V$. Choose $\bv_1,\dots,\bv_d \in M$ whose images in $V$ form a basis of $T$; 
the images of $\bv_1,\dots,\bv_d$ in $M \otimes_A \kappa$ are linearly independent, so
by Lemma~\ref{L:valuation lattice}, $\bv_1,\dots,\bv_d$ project to a basis of $M \otimes_A \frako$.
It follows that $\bv_1,\dots,\bv_d$ also project to a basis of $M \otimes_A A/(\pi^n)$ for each positive integer $n$.

Again by considering the dual bundle, we see that the image of $F_0$ in $M[\pi^{-1}]$ contains $\pi^n M$ for any sufficiently large integer $n$.
Let $\be_1,\dots,\be_m$ be the images in $M$ of the chosen basis of $F_0$;
using the previous paragraph, we can find elements $\be'_1,\dots,\be'_m \in A\bv_1 + \cdots + A\bv_d$ such that $\be'_j = \sum_i X_{ij} \be_i$ for some matrix $X$ over $A$ with
$\det(X) - 1 \in \pi A \subset \frakp$. The matrix $X$ is then invertible, whence $\bv_1,\dots,\bv_d$ generate $M[\pi^{-1}]$. By Lemma~\ref{L:transfer generators local},
$\bv_1,\dots,\bv_d$ generate $M$, necessarily freely.
\end{proof}

\begin{theorem} \label{T:algebraic glueing local}
The pullback functors
$\Vect_{X} \to \Vect_{Y} \to \Vect_{Z}$ are equivalences of categories.
\end{theorem}
\begin{proof}
By Lemma~\ref{L:algebraic glueing1 local}(a), 
the functors $\Vect_X \to \Vect_Y \to \Vect_Z$ are fully faithful,
so it suffices to check that $\Vect_X \to \Vect_Z$ is essentially surjective.
For $\calF \in \Vect_Z$, by Lemma~\ref{L:algebraic glueing3 local}, $M = H^0(Z, \calF)$ is a finite free $A$-module. By Lemma~\ref{L:algebraic glueing1 local}(c), we have $\tilde{M}|_Z \cong \calF$, proving the claim.
\end{proof}

\section{Adic glueing}
\label{sec:adic glueing}

We next show that vector bundles on $\Spec \bA_{\rinf}$ can be constructed by glueing not just for a Zariski covering, but for a covering in the setting of adic spaces; this result is used in \cite{s-berkeley} as part of the construction of mixed-characteristic local shtukas. In the process, we prove a somewhat more general result. Along the way, we will use results of Buzzard--Verberkmoes \cite{buzzard-verberkmoes}, Mihara \cite{mihara}, and Kedlaya--Liu \cite{kl1}.

We begin by summarizing various definitions from Huber's theory of adic spaces, as described in \cite{huber-book}. See also \cite[Lecture~1]{aws}.
\begin{defn}
We say that a topological ring $A$ is \emph{f-adic} if there exists an open subring $A_0$ of $A$ (called a \emph{ring of definition}) whose induced topology is the adic topology for some finitely generated ideal of $A_0$ (called an \emph{ideal of definition}). 
Such a ring is \emph{Tate} if it contains a topologically nilpotent unit; in certain cases (as in \cite[Lecture~1]{aws}), one may prefer to instead assume only that the topologically nilpotent elements generate the unit ideal, 
but we will not do this here.

We will only need to consider f-adic rings which are complete for their topologies,
which we refer to as \emph{Huber rings}. Beware that this definition is not entirely standard: some authors use the term \emph{Huber ring} as a synonym for \emph{f-adic ring} without the completeness condition.

For $A$ a Huber ring, let $A^\circ$ denote the subring of power-bounded elements of $A$;
we say that $A$ is \emph{uniform} if $A^\circ$ is bounded in $A$. (This implies that $A$ is reduced, but not conversely.)
A \emph{ring of integral elements} of $A$ is a subring of $A^\circ$ which is open and integrally closed in $A$.

A \emph{Huber pair} is a pair $(A,A^+)$ in which $A$ is a Huber ring and $A^+$ is a ring of integral elements of $A$. To such a pair, we may associate the topological space $\Spa(A,A^+)$ of equivalence classes of continuous valuations on $A$ which are bounded by 1 on $A^+$. This space may be topologized in such a way that 
a neighborhood basis is given by subspaces of the form
\[
\{v \in \Spa(A,A^+): v(f_1),\dots,v(f_n) \leq v(g) \neq 0\}
\]
for some $f_1,\dots,f_n,g \in A$ which generate an open ideal; such spaces are called
\emph{rational subspaces} of $\Spa(A,A^+)$. 
(When $A$ is Tate, every open ideal of $A$ is the unit ideal, and so the condition $v(g) \neq 0$ becomes superfluous.)
For this topology, $\Spa(A,A^+)$ is quasicompact and even a \emph{spectral space} in the sense of Hochster \cite{hochster}.

In addition, Huber defines a \emph{structure presheaf} $\calO$ on $\Spa(A,A^+)$; in the case where
$A$ is Tate and $U$ is the rational subspace defined by some parameters $f_1,\dots,f_n,g$, 
the ring $\calO(U)$ may be identified with the quotient $A \left \langle \frac{f_1}{g}, \dots, \frac{f_n}{g} \right\rangle$ of the Tate algebra $A \langle T_1,\dots,T_n \rangle$ 
by the closure of the ideal $(gT_1 - f_1, \dots, gT_n - f_n)$.

We say that $A$ is \emph{sheafy} if $\calO$ is a sheaf for some choice of $A^+$; with a bit of work
\cite[Remark~1.6.9]{aws}, the same is then true for any $A^+$. For example, by Proposition~\ref{P:buzzard-verberkmoes} below, this holds if $A$ is \emph{stably uniform}, meaning that 
(again for some, and hence any, choice of $A^+$) for every rational subspace $U$ of $\Spa(A,A^+)$,
the ring $\calO(U)$ is uniform.
\end{defn}

\begin{prop} \label{P:adic glueing}
Let $(A,A^+)$ be a Huber pair with $A$ Tate.
\begin{enumerate}
\item[(a)]
Choose $f \in A$ and suppose that
\[
0 \longrightarrow A \longrightarrow A \left\langle f \right\rangle \oplus A \left\langle f^{-1} \right\rangle \stackrel{(x,y) \mapsto x-y}{\longrightarrow} A \left\langle f^{\pm 1} \right\rangle \dashrightarrow 0
\]
is exact without the dashed arrow. (It is then also exact with the dashed arrow; e.g., see \cite[Lemma~1.8.1]{aws}.) Then the functor
\[
\Vect_{\Spec(A)} \to \Vect_{\Spec(A \langle f \rangle)} \times_{\Vect_{\Spec(A \langle f^{\pm 1} \rangle)}}
\Vect_{\Spec(A \langle f^{-1} \rangle)}
\]
is an equivalence of categories.
\item[(b)]
The conclusion of (a) holds whenever $A$ is (Tate and) uniform.
\item[(c)]
If $A$ is (Tate and) sheafy, then the pullback functor $\Vect_{\Spec(A)} \to \Vect_{\Spa(A,A^+)}$
is an equivalence of categories, with quasi-inverse given by the global sections functor.
\end{enumerate}
\end{prop}
\begin{proof}
For (a), see \cite[Lemma~1.9.12]{aws}.
For (b), see \cite[Corollary~2.8.9]{kl1} or \cite[Lemma~1.7.3, Lemma~1.8.1]{aws}.
For (c), see \cite[Theorem~2.7.7]{kl1} or \cite[Theorem~1.4.2]{aws}.
\end{proof}

Using Proposition~\ref{P:adic glueing}(a,b), one can deduce the following. However, we give references in lieu of a detailed argument.

\begin{prop}[Buzzard--Verberkmoes, Mihara] \label{P:buzzard-verberkmoes}
Any stably uniform Huber ring is sheafy.
\end{prop}
\begin{proof}
The original (independent) references are \cite[Theorem~7]{buzzard-verberkmoes} and \cite[Theorem~4.9]{mihara}. See also \cite[Theorem~2.8.10]{kl1}
or \cite[Theorem~1.2.13]{aws}.
\end{proof}

With these results in mind, we set some more specific notation.
\begin{hypothesis} \label{H:Huber ring}
For the remainder of \S\ref{sec:adic glueing}, let $R$ be a Huber ring which is perfect of characteristic $p$ and Tate, and let $R^+$ be a subring of integral elements in $R$
(which is necessarily also perfect). For example, we may take $R = K$, $R^+ = \frako_K$ in case $K$ is complete for a rank 1 valuation. 
Let $\overline{x} \in R$ be a topologically nilpotent unit; note that necessarily $\overline{x} \in R^+$.
\end{hypothesis}

For the geometric meaning of the following definition, see the proof of Theorem~\ref{T:algebraic to adic glueing}.
\begin{defn} \label{D:mess of rings}
Topologize 
\[
A_1 := W(R^+)[p^{-1}], 
\,
A_2 := W(R^+)[[\overline{x}]^{-1}],
\,
A_{12} := W(R^+)[(p[\overline{x}])^{-1}]
\]
as Huber rings with ring of definition $W(R^+)$ and ideals of definition generated by the respective topologically nilpotent units $p, [\overline{x}], p[\overline{x}]$. Then put
\[
B_1 := A_1 \left\langle \frac{[\overline{x}]}{p} \right\rangle,
\,
B_2 := A_2 \left\langle \frac{p}{[\overline{x}]} \right\rangle,
\,
B_{12} := A_{12} \left\langle \frac{[\overline{x}]}{p}, \frac{p}{[\overline{x}]} \right\rangle;
\]
note that there are canonical isomorphisms of topological rings
\[
B_{12} \cong B_1\left\langle \frac{p}{[\overline{x}]} \right\rangle
\cong
B_2 \left \langle \frac{[\overline{x}]}{p} \right\rangle.
\]
Also put
\[
B'_1 := A_2 \left \langle \frac{[\overline{x}]}{p} 
\right\rangle,
\,
B'_2 := A_1 \left \langle \frac{p}{[\overline{x}]} \right\rangle;
\]
note that there are canonical isomorphisms of underlying rings
\[
B'_1 \cong B_1 [[\overline{x}]^{-1}], \, B'_2 \cong B_2[p^{-1}]
\]
but these are not homeomorphisms for the implied topologies.
For example, in the first isomorphism, the rings of power-bounded elements coincide, but on this common subring the induced topology from $B'_1$
is the $\frac{[\overline{x}]}{p}$-adic topology while the induced topology from $B_1 [[\overline{x}]^{-1}]$
is the $p$-adic topology.
\end{defn}

\begin{prop} \label{P:all sheafy}
The following statements hold.
\begin{enumerate}
\item[(a)]
The Huber rings $C = A_1, A_{12}, B_1, B_2, B_{12}, B'_2$ are stably uniform, and hence sheafy
by Proposition~\ref{P:buzzard-verberkmoes}.
\item[(b)]
The Huber ring $C = A_2$ is uniform. (The same is true for $C = B'_1$, but we will not need this. See also
Remark~\ref{R:not stably uniform}.)
\end{enumerate}
\end{prop}
\begin{proof}
To prove (a), note that for $C = A_1, A_{12}, B_1, B_{12}, B'_2$, 
$p$ is a topologically nilpotent unit in $C$.
In these cases, by \cite[Theorem~5.3.9]{kl1},
taking the completed tensor product over $\ZZ_p$ with $\ZZ_p[p^{p^{-\infty}}]$ yields a perfectoid ring in the sense of \cite{kl1} (which must be a $\QQ_p$-algebra). By splitting from $\ZZ_p[p^{p^{-\infty}}]$ to $\ZZ_p$ using the reduced trace,
we deduce that $C$ is stably uniform; see \cite[Theorem~3.7.4]{kl1} for further details.
For $C = B_2$, $p$ is no longer a unit in $C$ but is still topologically nilpotent,
and a similar argument applies using perfectoid rings in the sense of Fontaine; see \cite[Corollary~4.1.14]{kl2}
or \cite[Lemma~3.1.3]{aws}.

To prove (b), note that $A_2^\circ$ is $p$-adically saturated in $A_2$,
$W(R^\circ)$ is contained in $A_2^\circ$, and the image of $A_2^\circ/(p) \to A_2/(p) \cong R$ is contained in $R^\circ$.
These facts together imply that $A_2^\circ = W(R^\circ)$, which is evidently a
bounded subring of $A_2$.
\end{proof}

\begin{remark} \label{R:not stably uniform}
We believe that $A_2$ is stably uniform, which would then imply that $B'_1$ is stably uniform; but we were unable to prove either of these statements.
One thing we can observe is that if $B'_1$ were known to be stably uniform, then combining the preceding results with
Proposition~\ref{P:adic glueing}(a) and \cite[Theorem~1.2.22]{aws} would imply that $A_2$ is sheafy (and then stably uniform).
\end{remark}

We now obtain a comparison between algebraic and adic vector bundles.
\begin{theorem} \label{T:algebraic to adic glueing}
Put $A := W(R^+)$ and let $X$ (resp.\ $Y$) be the complement in $\Spec A$ (resp.\ $\Spa(A,A)$) of the closed subspace where $p = [\overline{x}] = 0$. 
Then pullback along the morphism $Y \to X$ of locally ringed spaces defines an equivalence of categories
$\Vect_X \to \Vect_Y$.
\end{theorem}
\begin{proof}
For $A_1, A_2, A_{12}, B_1, B_2, B_{12}, B'_1, B'_2$ as in Definition~\ref{D:mess of rings},
we have the following coverings of adic spaces by rational subspaces.
\[
\begin{array}{c|c|c|c}
U \cup V & U & V & U \cap V \\
\hline
Y & \Spa(B_1, B_1^\circ) & \Spa(B_2, B_2^\circ) & \Spa(B_{12}, B_{12}^\circ) \\
\Spa(A_1, A_1^\circ) & \Spa(B_1, B_1^\circ) & \Spa(B'_2, B_2^{\prime \circ}) &
\Spa(B_{12}, B_{12}^\circ) \\
\Spa(A_2, A_{2}^\circ) & \Spa(B'_1, B_1^{\prime \circ}) & \Spa(B_2, B_2^{\circ}) &
\Spa(B_{12}, B_{12}^\circ) \\
\Spa(A_{12}, A_{12}^\circ)  & \Spa(B'_1, B_1^{\prime \circ}) & \Spa(B'_2, B_2^{\prime \circ}) & \Spa(B_{12}, B_{12}^\circ)
\end{array}
\]
For $i \in \{1, 2, 12\}$, we may apply Proposition~\ref{P:adic glueing}(c) and Proposition~\ref{P:all sheafy}(a) to see that
the pullback functor $\Vect_{\Spec(B_i)} \to \Vect_{\Spa(B_i, B_i^\circ)}$ is an equivalence.
We may also apply Proposition~\ref{P:adic glueing}(a,b) and Proposition~\ref{P:all sheafy}(b) to obtain
an equivalence
\[
 \Vect_{\Spec(A_i)} \to \Vect_{\Spec(B_1^{?})} \times_{\Vect_{\Spec(B_{12})}}
\Vect_{\Spec(B_2^?)}, \qquad
B_j^? = \begin{cases} B_j & j \in i \\ B_j' & j \notin i; \end{cases}
\]
using the fact that $A_j \to B'_j$ factors through $B_j$ (at the level of rings without topology), it follows that
\[
\Vect_{\Spec(A_1)} \times_{\Vect_{\Spec(A_{12})}}
\Vect_{\Spec(A_2)}
\to
\Vect_{\Spec(B_1)} \times_{\Vect_{\Spec(B_{12})}}
\Vect_{\Spec(B_2)}
\]
is an equivalence.
In the 2-commutative diagram
\[
\xymatrix@C30pt{
\Vect_X \ar[r] \ar@{=}[d] & \Vect_Y \ar@{=}[dd] \\
\Vect_{\Spec(A_1)} \times_{\Vect_{\Spec(A_{12})}} \Vect_{\Spec(A_2)} \ar[d] & \\
\Vect_{\Spec(B_1)} \times_{\Vect_{\Spec(B_{12})}} \Vect_{\Spec(B_2)} \ar[r] &
\Vect_{\Spa(B_1,B_1^\circ)} \times_{\Vect_{\Spa(B_{12},B_{12}^\circ)}} \Vect_{\Spa(B_2,B_2^\circ)}
}
\]
 every arrow except $\Vect_X \to \Vect_Y$ is now known to be an equivalence; we thus obtain the desired result.
\end{proof}

As a corollary, we obtain the following theorem.

\begin{theorem} \label{T:adic glueing local}
Let $v_0$ be the valuation on $W(\frako_K)$ induced by the trivial valuation on the residue field of $\frako_K$.
Put $A := W(\frako_K)$, $X := \Spa(A,A)$, $Y := X \setminus \{v_0\}$. 
Let $\Mod^{\ff}_A$ be the category of finite free $A$-modules. Then the categories 
$\Mod^{\ff}_A, \Vect_X, \Vect_Y$ are equivalent via the functor $\Mod^{\ff}_A \to \Vect_X$ taking $M$ to $\tilde{M}$, the pullback functor $\Vect_X \to \Vect_Y$, and the global sections functor $\Vect_Y \to \Mod^{\ff}_A$. 
\end{theorem}
\begin{proof}
Combine Theorem~\ref{T:algebraic glueing local}
with Theorem~\ref{T:algebraic to adic glueing}.
\end{proof}

One might like to parlay Theorem~\ref{T:adic glueing local} into a version with $K$ replaced by $R$. However, one runs into an obvious difficulty in light of the following standard example in the category of schemes.
\begin{example} \label{exa:S2 module}
Let $k$ be a field, put $S := k[x,y,z]$, and let $M$ be the $S$-module
\[
\ker(S^3 \to S: (a,b,c) \mapsto ax+by+cz).
\]
Put $X := \Spec S$, $Y := X \setminus \{(x,y,z)\}$,
$Z := X \setminus \overline{\{(x,y)\}}$; then $\tilde{M} \notin \Vect_X$
but $\tilde{M}|_* \in \Vect_*$ for $* \in \{Y,Z\}$.
Since $X \setminus Y$ has codimension 3 in $X$ and $Y\setminus Z$ has codimension 2 in $Y$, $\tilde{M}|_Z$ has a unique extension to an $S_2$ sheaf (in the sense of Serre)
on either $X$ or $Y$,
namely $\tilde{M}$ itself. In particular, $\tilde{M}$ does not lift from $\Vect_Z$ to $\Vect_X$.
\end{example}

With a bit of care, this argument can be translated into an example that shows that Theorem~\ref{T:adic glueing local} indeed fails to generalize to the case where $K$ is replaced by $R$.

\begin{remark} \label{R:no algebraic glueing}
For $(R,R^+)$ as in Hypothesis~\ref{H:Huber ring}, 
let $\frakp$ be the radical of the ideal $(p, [\overline{x}])$;
it is generated by $p$ and $[\overline{x}]^{p^{-n}}$ for all $n$.
Put
\[
X := \Spec(W(R^+)), \qquad
Y := X \setminus \{\frakp\},
\]
and let $Z$ be the algebraic stack which is the colimit of the diagram
\[
\Spec(W(R^+)[p^{-1}]) \leftarrow \Spec(W(R)[p^{-1}]) \to \Spec(W(R)).
\]
As in Lemma~\ref{L:algebraic glueing1 local}, we see that the functors
$\Vect_X \to \Vect_Y, \Vect_Y \to \Vect_Z$ are fully faithful, and that for $* \in \{Y,Z\}$, $\calF \in \Vect_*$, $M = H^0(*, \calF)$, the adjunction morphism $\tilde{M}|_* \to \calF$ is an isomorphism. However, one may emulate Example~\ref{exa:S2 module} so as to produce an object of $\Vect_Y$  and $\Vect_Z$ which does not lift to $\Vect_X$;
see Example~\ref{exa:no algebraic glueing} below.
\end{remark}

\begin{lemma} \label{L:reflexive}
With notation as in Remark~\ref{R:no algebraic glueing}, for $\calF \in \Vect_*$ and $M = H^0(*, \calF)$, the natural homomorphism
$M^\vee \to H^0(*, \calF^{\vee})$ is an isomorphism. Consequently, the map $M \to M^{\vee \vee}$ is an isomorphism, i.e., $M$ is reflexive.
\end{lemma}
\begin{proof}
From Remark~\ref{R:no algebraic glueing}, we see that the map is injective. To check surjectivity, note that any $f \in H^0(*, \calF^{\vee})$ restricts to maps $M \to W(R^+)[p^{-1}]$,
$M \to W(R)$ which induce the same map $M \to W(R)[p^{-1}]$. We again deduce the claim from the equality
$W(R^+)[p^{-1}] \cap W(R) = W(R^+)$.
\end{proof}

\begin{remark} \label{R:regular sequence}
Recall that for any ring $S$, a \emph{regular sequence} in $S$ is a finite sequence $s_1,\dots,s_k$ such that for $i=1,\dots,k$, $s_i$ is not a zero-divisor in $S/(s_1,\dots,s_{i-1})$. If $s_1,\dots,s_k$ is a regular sequence in $S$, one computes easily that
\[
\Tor_k^S(S/(s_1,\dots,s_k), S/(s_1,\dots,s_k)) \cong S/(s_1,\dots,s_k) \neq 0;
\]
in particular,  $S/(s_1,\dots,s_k)$ has projective dimension at least (and in fact exactly) $k$ as an $S$-module.
\end{remark}

\begin{example} \label{exa:no algebraic glueing}
Let $k$ be a perfect field of characteristic $p$.
Let $R^+$ be the $(\overline{y}, \overline{z})$-adic completion of the 
perfect closure of $k\llbracket \overline{y}, \overline{z} \rrbracket$.
Put $\overline{x} := \overline{y} \overline{z} \in R^+$.
This notation is consistent with Hypothesis~\ref{H:Huber ring},
so we may adopt notation as in Remark~\ref{R:no algebraic glueing}.

Put $I := ([\overline{y}], [\overline{z}], p)W(R^+)$; note that the generators of $I$ form a regular sequence.
By Remark~\ref{R:regular sequence}, $W(R^+)/I$ has projective dimension at least 3, $I$ has projective dimension at least 2, and
\[
M := \ker(W(R^+)^3 \to I: (a,b,c) \mapsto a[\overline{y}] + b[\overline{z}] + cp)
\]
has projective dimension at least 1. In particular, $M$ is not projective.

For $* \in \{Y,Z\}$, the sequence
\begin{equation} \label{eq:bad exact}
0 \to \tilde{M}|_* \to \calO^{\oplus 3} \to \calO \to 0
\end{equation}
of sheaves is exact, so $\tilde{M}|_* \in \Vect_*$. 
Because $H^0(*, \calO) = W(R^+)$, applying the functor
$H^0(*, \bullet)$ to \eqref{eq:bad exact} yields an isomorphism $H^0(*, \tilde{M}|_*) \cong M$.

However, if $\tilde{M}|_*$ could be extended to an object $\calF \in \Vect_X$, 
we would have $\calF \cong \tilde{N}$ for some finite projective $W(R^+)$-module $N$, and per Remark~\ref{R:no algebraic glueing} we would have $N \cong H^0(*, \calF) = H^0(*, \tilde{M}|_*) \cong M$. This yields a contradiction.
\end{example}

\end{document}